\documentclass{amsart}
\usepackage{mathtools}
\usepackage{enumerate}
\usepackage{amssymb}
\usepackage{amsmath}
\usepackage{amsthm}
\usepackage{enumitem}
\usepackage{hyperref}
\hypersetup{
           breaklinks=true,   
           colorlinks=true,   
        }
\numberwithin{equation}{section}
\theoremstyle{plain}
\newtheorem{thm}[equation]{Theorem}

\newtheorem{prop}[equation]{Proposition}
\newtheorem{lem}[equation]{Lemma}
\newtheorem*{claim*}{Claim}
\theoremstyle{definition}
\newtheorem{defn}[equation]{Definition}

\theoremstyle{remark}
\newtheorem{remark}[equation]{Remark}
\newtheorem{asmp}[equation]{Assumption}
\newcommand{\delb}{\sqrt{-1}\partial\bar{\partial}}
\newcommand{\relmiddle}[1]{\mathrel{}\middle#1\mathrel{}}
\usepackage{xparse}
\usepackage{bbm}
\author[R.\ Murakami]{Rei Murakami}
\address{Mathematical Institute, Tohoku University, 6-3, Aramaki Aza-Aoba, Aoba-ku, Sendai 980-8578, Japan}
\email{rei.murakami.p3@dc.tohoku.ac.jp, reimurakami66@gmail.com}

\begin{document}

\title[Weak limits of the $J$-flow and the dHYM flow]{Weak limits of the $J$-flow and the deformed Hermitian-Yang-Mills flow on K{\"a}hler surfaces: boundary cases}

\begin{abstract}
    We prove that if a pair of K{\"a}hler classes is $J$-nef, the $J$-flow on a compact K{\"a}hler surface converges to a weak solution of the Monge-Amp{\`e}re equation in the sense of currents. We also establish the same convergence behavior for the deformed Hermitian-Yang-Mills flow. The method is based on a property of a limit of viscosity subsolutions.
    
\end{abstract}

\maketitle


\section{Introduction}

The study of geometric flows has emerged as a powerful tool in understanding the complex structure and dynamics of K{\"a}hler manifolds. They provide a geometric process by which one can deform a K{\"a}hler manifold in a controlled manner, aiming to achieve a certain optimality. A critical question in the study is its long-time behavior, particularly its convergence properties, which are naturally related to a canonical structure on K{\"a}hler manifolds.
This paper delves into the analysis of the convergence behavior of two geometric flows, the $J$-flow and the deformed Hermitian-Yang-Mills flow. 

Let us consider an $n$-dimensional K{\"a}hler manifold $X$ endowed with two K{\"a}hler forms, $\chi$ and $\omega$. The $J$-flow, introduced by Donaldson \cite{Donaldson} to study a certain moment map, is a PDE of $\chi$-psh functions given by
\begin{equation*}
        \begin{cases}
            \dot{\varphi_t}=c-\Lambda_{\chi_t}\omega, \\
            \varphi|_{t=0}=\varphi_0,
        \end{cases}
\end{equation*}
where $\dot{\varphi_t}$ denotes the time derivative of $\varphi_t$, the constant $c$ is determined by 
$$c=\frac{\int_X n \, \omega\wedge\chi^{n-1}}{\int_X \chi^n},$$
and $\Lambda_{\chi_t}$ denotes the trace operator with respect to $\chi_t=\chi+\delb\varphi_t$. 
A stationary point of the $J$-flow satisfies the $J$-equation, 
\begin{equation*}
    \Lambda_\chi\omega=c,
\end{equation*}
 discovered by Xiuxiong Chen independently in \cite{XChen} from the perspective of constant scalar curvature K{\"a}hler metrics.
As in \cite{Song}, 
we use the following terminologies:
\begin{defn}
    A pair of K{\"a}hler classes $([\chi],[\omega])$ is called 
    \begin{enumerate}
        \item \textit{$J$-positive} if for any $p \in \{1,\dots,n-1\}$ and any $p$-dimensional subvariety $V \subset X$, 
        we have
        $$\frac{\int_V p \, \omega\wedge\chi^{p-1}}{\int_V \chi^p} < c.$$
        \item \textit{$J$-nef} if for any $p \in \{1,\dots,n-1\}$ and any $p$-dimensional subvariety $V \subset X$, 
        we have
        $$\frac{\int_V p \, \omega\wedge\chi^{p-1}}{\int_V \chi^p} \le c.$$
    \end{enumerate}
\end{defn}
The solvability of the $J$-equation is now known to be equivalent to the $J$-positivity of $([\chi],[\omega])$ \cite{GChen,Datar-Pingali,Song}. On the other hand, \cite{XChen2} showed that the $J$-flow exists for all time, and \cite{SW} showed that the $J$-flow converges smoothly to the solution of the $J$-equation if it exists.

In this paper, we investigate the long-time behavior of the $J$-flow of a $J$-nef pair.

\begin{thm}\label{thmJmain}
    Let $X$ be a compact K{\"a}hler surface. Let $\chi$ and $\omega$ be K{\"a}hler forms and assume $([\chi],[\omega])$ is $J$-nef. Then the $J$-flow $\{\chi_t\}$ converges, in the sense of currents, to the unique current $\chi_\infty$ such that $c\chi_\infty-\omega$ is a positive current and
    \begin{equation}\label{eqJMA}
        \langle (c\chi_\infty-\omega)^2 \rangle=\omega^2,
    \end{equation}
    where $\langle\cdot\rangle$ denotes the non-pluripolar product.
\end{thm}

\begin{remark}
\begin{itemize}
    \item In dimension 2, we can see that a pair $([\chi],[\omega])$ is $J$-nef if and only if the class $[c\chi-\omega]$ is nef and big (the bigness follows from $[c\chi-\omega]^2>0$ and \cite[Theorem 0.5]{DP}). Therefore, the existence and the uniqueness of the solution of \eqref{eqJMA} are proved in \cite{BEGZ}. What we prove here is that the $J$-flow converges to the solution.
    \item Thoerem \ref{thmJmain} solves one of the conjectures posed in \cite{DMS} (see below Conjecture 1.5 there) on Kähler surfaces in $J$-nef cases.
\end{itemize}
\end{remark}

The theorem was proved in \cite{FLSW} when $[c\chi-\omega]$ is additionally semipositive. In this case, they, moreover, proved that the $J$-flow converges smoothly outside a finite number of curves. 

The strategy of the proof of Theorem \ref{thmJmain} is decomposed into 3 steps. 
\begin{enumerate}[label={Step \arabic*}]
    \item\label{stepslope} \textit{$L^2$-convergence of the time derivative.} We prove that the time derivative or equivalently $c-\Lambda_\chi \omega$ converges to zero in $L^2(\omega)$. The key is the convexity of the $\mathcal{J}$-functional along the $J$-flow.
    \item\label{stepvsub} \textit{Viscosity Subsolution.} We prove that the $L^2$-convergence implies that the limit is a viscosity subsolution in the sense of Definition \ref{defnvsub}. The strategy has already appeared in \cite[Proposition 21]{CS17}. The key is a property of a limit in the viscosity theory (Proposition \ref{propCIL}).
    \item\label{steppp} \textit{From subsolution to solution.} We observe that a viscosity subsolution is a subsolution as a measure. In dimension 2, the $J$-equation is equivalent to the Monge-Amp{\`e}re equation, and the statement was claimed in \cite[Corollary 2.6]{EGZ}. Since we could not find proof, we provide it in this paper (Lemma \ref{lemEGZ}). We use here the facts that the solution of \eqref{eqJMA} exists and that it has minimal singularities. With the converse inequality from \cite[Proposition 1.20]{BEGZ}, we complete the proof.
\end{enumerate}

We also apply the method to the deformed Hermitian-Yang-Mills (dHYM for short) flow. Consider an $n$-dimensional compact K{\"a}hler manifold $X$ endowed with two smooth closed real $(1,1)$-forms $\alpha$ and $\omega$, where $\omega$ is K{\"a}hler. The dHYM flow, introduced in \cite{FYZ}, is defined by
\begin{equation}\label{eqdhymflow}
    \begin{cases}
        \dot{\varphi}_t=\cot{\theta_\omega(\alpha_t)}-\cot{\theta_0},\\
        \varphi|_{t=0}=\varphi_0,
    \end{cases}
\end{equation}
where $\theta_\omega(\chi)=\sum_i \mathrm{arccot}\, \lambda_i$ and $\lambda_i$'s are eigenvalues of $\omega^{-1}\alpha$. The cotangent function $\cot{\theta}$ is defined by $1/{\tan{\theta}}$.
The constant $\theta_0$ is chosen to satisfy
$$\theta_0=\arg{\int_X(\alpha+\sqrt{-1}\omega)^n}.$$
The dHYM flow is designed to study the dHYM equation
$${\theta_\omega(\alpha)}={\theta_0},$$
 originally investigated by \cite{JY}, inspired by concepts in mirror symmetry as discussed in \cite{LYZ}. 
 In this paper, we assume $0 <\theta_0 < \pi$, called the supercritical condition. For the solvability of the supercritical dHYM equation, \cite{GChen, CLT} proved statements analogous to \cite{GChen, Song}. By the same strategy stated above, we prove the following:
\begin{thm}\label{thmdHYMmain}
    Let $X$ be a compact K{\"a}hler surface. Let $\alpha$ and $\omega$ be closed real $(1,1)$-forms, where $\omega$ is K{\"a}hler. Assume $0<\theta_0<\pi$ and $[\alpha-\cot{\theta_0}\omega]$ is nef. Assume that an initial form of the dHYM flow satisfies $0<\theta_\omega(\alpha_0)<\pi$. Then the dHYM flow $\{\alpha_t\}$ converges, in the sense of currents, to the unique current $\alpha_\infty$ such that $\alpha_\infty-\cot{\theta_0}\omega$ is a positive current and
    \begin{equation*}
        \langle (\alpha_\infty-\cot{\theta_0}\omega)^2\rangle=(\cot^2{\theta_0}+1)\omega^2.
    \end{equation*}
\end{thm}

Again, the bigness of $[\alpha-\cot\theta_0\omega]$ follows from $[\alpha-\cot\theta_0\omega]^2>0$ and \cite[Theorem 0.5]{DP}.
The theorem was proved in \cite{FYZ} when $[\alpha-\cot{\theta_0}\omega]$ is additionally semipositive. In this case, they, moreover, proved that the dHYM flow converges smoothly outside a finite number of curves.

The paper is organized as follows. In Subsection \ref{secslope}, we conduct \ref{stepslope}. In Subsection \ref{secvsub}, we conduct \ref{stepvsub}. In Subsection \ref{secpp}, we conduct \ref{steppp}. In Section \ref{secdHYM}, we prove Theorem \ref{thmdHYMmain}. We conduct \ref{stepslope} in detail and only point out the differences from the $J$-flow on \ref{stepvsub} and \ref{steppp}.

\begin{subsection}*{Acknowledgements}
The author thanks his supervisor Shin-ichi Matsumura for assistance with revising the paper, as well as for encouragement and support. He also thanks Ryosuke Takahashi for answering questions. 
He further thanks the anonymous referee for many helpful comments and suggestions.
This work was supported by JST, the establishment of university fellowships towards the creation of science technology innovation, Grant Number JPMJFS2102, and by JSPS KAKENHI Grant Number JP24KJ0346.
\end{subsection}

\section{The \texorpdfstring{$J$}{J}-flow}
\subsection{The asymptotic slope of the \texorpdfstring{$\mathcal{J}$}{J}-functional along the \texorpdfstring{$J$}{J}-flow}\label{secslope}
In this subsection, we prove that the time derivative of the $J$-flow converges to zero in the $L^2$-norm if the pair is $J$-nef. 
We start with the observation that the $\mathcal{J}$-functional, given as 
$$\mathcal{J}(0)=0,\quad d\mathcal{J}(\varphi)(\psi)=\int_X \psi\, \left(\Lambda_{\chi_\varphi}\omega-c\right)\chi_\varphi^n,$$
is convex along the $J$-flow.
The following lemma is key to proving it. 

\begin{lem}[{\cite[Lemma 1]{Hashimoto}}]\label{lemhsmt}
    Let $F_{\chi,\omega}$ be the operator defined by
    $$F_{\chi,\omega}(\varphi) = (\delb \varphi, \omega)_\chi + (\partial \Lambda_\chi \omega, \bar{\partial} \varphi)_\chi.$$
    Then, we have
    $$ \int_X \varphi \, F_{\chi,\omega}(\varphi) \, \chi^n \le 0$$
    for any $\varphi \in C^\infty(X)$.
\end{lem}

\begin{lem}\label{lemconvexity}
    The $\mathcal{J}$-functional is convex along the $J$-flow.
\end{lem}
\begin{proof}[Proof of Lemma \ref{lemconvexity}] We can compute as
    \begin{align*}
        \frac{d^2}{dt^2}\mathcal{J}(\varphi_t)
        &= - \frac{d}{dt} \int_X \dot{\varphi}_t^2 \, \chi_t^n
        = - \int_X 2 \, \dot{\varphi}_t \, \ddot{\varphi_t} \, \chi^n_t 
          - \int_X \dot{\varphi}_t^2 \, n \, \delb \dot{\varphi_t}  \wedge \chi^{n-1}_t \\
        &= - \int_X 2 \, \dot{\varphi}_t \, \ddot{\varphi_t} \, \chi^n_t 
           + \int_X 2 \, \dot{\varphi}_t \, n \, \sqrt{-1} \partial \dot{\varphi}_t\wedge \bar{\partial} \dot{\varphi}_t \wedge \chi^{n-1}_t \\
        &= - \int_X 2\, \dot{\varphi}_t \, (\delb \dot{\varphi}_t, \omega)_{\chi_t} \, \chi^n_t
           - \int_X 2 \, \dot{\varphi}_t \, n \, \sqrt{-1} \partial \Lambda_{\chi_t} \omega \wedge \bar{\partial} \dot{\varphi}_t \wedge \chi^{n-1}_t \\
        &= -2 \int_X \dot{\varphi}_t \, F_{\chi_t,\omega}(\dot{\varphi}_t) \, \chi^n_t.
    \end{align*}
    Therefore, the claim follows by Lemma \ref{lemhsmt}.
\end{proof}

We are now in a position to establish the convergence in $L^2(\omega)$. The statement readily follows if $J$-nefness corresponds to the $\mathcal{J}$-functional being bounded from below, given the nature of the $J$-flow as a gradient flow of the $\mathcal{J}$-functional. While such a correspondence appears plausible, proof remains elusive. The best achievement to date is \cite[Theorem 23]{SD}, which suffices for validating the statement.

\begin{prop}\label{propslope}
    If a pair of K{\"a}hler classes $([\chi],[\omega])$ is $J$-nef, then the time derivative of the $J$-flow converges to zero in $L^2(\omega)$.
\end{prop}
\begin{proof}
    Note that 
    \begin{equation}\label{eqslope}
        \int_X \dot{\varphi}_t^2 \, \chi_t^n= -\frac{d}{dt}\mathcal{J} (\varphi_t) \ge 0.
    \end{equation}
    Since the $\mathcal{J}$-functional is convex along the $J$-flow by Lemma \ref{lemconvexity}, we know that there exists a constant $\mu \le 0$ such that
    $$\lim_{t\rightarrow\infty}\frac{d}{dt}\mathcal{J}(\varphi_t) = \mu.$$
    We claim that $\mu =0.$ Since the pair $([\chi],[\omega])$ is $J$-nef, by \cite[Theorem 23]{SD}, we see that for any $\varepsilon>0$, there exists a constant $C_\varepsilon$ such that for any $\varphi \in C^\infty \cap PSH(X,\chi)$,
    $$\mathcal{J} (\varphi) \ge -\varepsilon \, \mathrm{E}^\chi_\chi(\varphi) - C_\varepsilon,$$
    where $\mathrm{E}^\chi_\chi$ is defined by
    $$\mathrm{E}^\chi_\chi(0)=0,\quad d\mathrm{E}^\chi_\chi(\varphi)(\psi)=\int_X \psi \, (n-\Lambda_{\chi_\varphi}\chi) \chi_\varphi^n.$$
    Along the $J$-flow, we have
    $$\int_X \dot{\varphi}_t \, \chi_{\varphi_t}^n=0.$$ 
    The maximum principle for the $J$-flow implies that
    $|\dot{\varphi}_t| < C$ 
    for some constant C. 
    Thus, we get
    $\mathrm{E}^\chi_\chi(\varphi_t) \le A+C \, t$ 
    for some constant $C$, where $A=E^\chi_\chi(\varphi_0)$.
    Combining these estimates, we have
    $$\mathcal{J}(\varphi_t) \ge - C \varepsilon t - C_\varepsilon-A\varepsilon.$$
    By dividing both hands by $t$ and letting $t$ tend to $\infty$, we have 
    $\mu \ge -C \varepsilon.$
    Since $\varepsilon>0$ can be arbitrary and $\mu \le 0$, we get $\mu =0.$ Since the maximum principle implies that $|\dot{\varphi_t}|<C$ for some constant $C$, we also know that 
    $\chi_t > C \omega$
    for some constant $C>0$.
    Therefore, by \eqref{eqslope}, the equality $\mu=0$ confirms the statement. 
\end{proof}

\subsection{Viscosity subsolutions}\label{secvsub}
For positive definite Hermitian $n \times n$ matrices $A$ and $B$, we define 
$$P_B(A) = \max_{k=1,\dots,n} \left(\sum_{j \neq k} \frac{1}{\lambda_j}\right), \quad 
Q_B(A) = \sum^n_{j=1} \frac{1}{\lambda_j},$$
where $\lambda_j$'s are eigenvalues of the matrix $B^{-1}A$.
For K{\"a}hler forms $\chi$ and $\omega$, we define
$$P_\omega(\chi) = P_{I_n}(\omega^{-1}\chi), \quad
Q_\omega(\chi) = Q_{I_n}(\omega^{-1}\chi),$$
where $\omega^{-1}\chi$ is seen as a matrix. By Proposition \ref{propslope}, along the $J$-flow $\varphi_t$, we have
$$\|Q_\omega(\chi_t)-c\|_{L^2(\omega)} \rightarrow 0.$$
We prove that this convergence implies that the weak limit is a viscosity subsolution. We first recall the definition of the viscosity subsolution in a domain (see \cite{CIL, DDT}).

\begin{defn}
    Let $\Omega \subset \mathbb{C}^n$ be a domain and $A$ an $n\times n$-matrix-valued function which is positive definite and Hermitian for any $x\in\Omega$. An upper semicontinuous function $\varphi$ is said to be a viscosity subsolution of $Q_{A}(\delb \varphi) = c$ at $x \in \Omega$, denoted by 
    $$Q_{A}(\delb\varphi)(x)\le_v c,$$ 
    if for any upper test $q$ of $\varphi$ at $x \in \Omega$ (i.e.,\,a function $q$ is $C^2$-function defined on a neighborhood $U$ of $x$ and satisfies $q\ge\varphi$ with $q(x)=\varphi(x)$), we have $$Q_{A(x)}\left((\delb q)(x)\right) \le c.$$
    An upper semicontinuous function $\varphi$ is said to be a viscosity subsolution in $\Omega$, denoted by
    $$Q_{A}(\delb \varphi) \le_v c,$$
    if $Q_A(\delb \varphi) (x)\le_v c$ at any point $x \in\Omega$.
\end{defn}

We give the following definition of viscosity subsolutions on K{\"a}hler manifolds. For the Monge-Amp{\`e}re equation, the definition has appeared in \cite{EGZ}.
\begin{defn}\label{defnvsub}
    Let $(X,\omega)$ be a K{\"a}hler manifold. We say that a closed positive $(1,1)$-current $\chi$ is a viscosity subsolution, denoted by
    $$Q_\omega(\chi) \le_v c,$$
    if for any point $x \in X$ and any coordinate neighborhood $U$ of $x$, we have $$Q_\omega(\delb\varphi)(x) \le_v c,$$
    where $\varphi$ is an upper semicontinuous local potential of $\chi$.
\end{defn}

Consider $\psi_t := \varphi_t - \sup{\varphi_t}$. By weak compactness, we have an $L^1$-limit $\psi_\infty$ of $\psi_t$ with $\sup{\psi_\infty}=0$ if we take a subsequence. To prove $Q_\omega(\chi_\infty)\le_v c$, where $\chi_\infty=\chi+\delb\psi_\infty$, the following proposition on the limit behavior in the viscosity theory is essential:

\begin{prop}[{\cite[Proposition 4.3]{CIL}}]\label{propCIL}
    Let $\Omega \in \mathbb{C}^n$ be a domain, $v$ an upper semicontinuous function on $\Omega$, $z \in \Omega$, and $q$ an upper test of $v$ at $z$. Suppose also that $u_n$ is a sequence of upper semicontinuous functions on $\Omega$ such that
    \begin{enumerate}[label={\upshape(\roman*)}]
        \item\label{itemCIL1} there exists $x_n \in \Omega$ such that $(x_n,u_n(x_n)) \rightarrow (z,v(z))$,
        \item\label{itemCIL2} if $z_n \in \Omega$ and $z_n \rightarrow x\in\Omega$, then $\limsup_{n\rightarrow\infty}u_n(z_n)\le v(x)$.
    \end{enumerate}
    Then there exist $\hat{x}_n\in \Omega$ and upper tests of $u_n$ at $\hat{x}_n$ denoted by $q_n$ such that
    $$(\hat{x}_n,u_n(\hat{x}_n),Dq_n(\hat{x}_n),D^2q_n(\hat{x}_n))\rightarrow(z,v(z),Dq(z),D^2q(z)).$$
\end{prop}

\begin{prop}\label{propvsub}
    Suppose that a sequence of K{\"a}hler forms $\chi_k=\chi+\delb\psi_k$ satisfies
    \begin{align*}
        \psi_k \rightarrow \psi_\infty \, \text{in $L^1(\omega)$ and} \ 
        \|Q_\omega(\chi_k)-c\|_{L^2(\omega)} \rightarrow 0.    
    \end{align*}
    Then, we have
    $$Q_\omega(\chi_\infty)\le_v c,$$
    where $\chi_\infty =\chi+\delb\psi_\infty$.
\end{prop}

\begin{proof}[Proof of Proposition \ref{propvsub}]
    We follow the proof of \cite[Proposition 21]{CS17}. 
    Fix a point $p \in X$ and take a coordinate neighborhood $U$ of $p$. Let $\omega_0$ be a K{\"a}hler form on some neighborhood $U_0 \subset U$ with constant coefficients such that $\omega_0 \le \omega$. Denote by $f$ an upper semicontinuous local potential of $\chi$ on $U_0$. Take a standard mollifier $\rho$, i.e.,\,a function $\rho$ satisfies $$\rho\ge0, \ \rho(w)=0 \text{ for $|w|>1$},\ \text{and } \int\rho(w)dw=1.$$ Define $u_k=f+\psi_k$ and denote by $u^\delta_k$ the regularization of $u_k$, i.e.,
    \begin{equation*}
        u^\delta_k(z)=\int_{\mathbb{C}^n} u_k(z-\delta w)\rho(w)dw.
    \end{equation*}
    The function $u^\delta_k$ is defined on $U_{0,\delta} =\{z\in U_0\mid d(z,\partial U_0)> \delta \}.$
    We see that for $z \in U_{0,\delta}$,
    \begin{equation}\label{eqregularization}
        \begin{split}
        Q_{\omega_0}(\delb u^\delta_k)(z)
        &=Q_{I_n}\left(\omega^{-1}_0(z) \int_{\mathbb{C}^n}\delb u_k(z-\delta w)\rho(w)dw
        \right)\\
        &=Q_{I_n}\left( \int_{\mathbb{C}^n}(\omega^{-1}_0 \chi_k)(z-\delta w)\rho(w)dw
        \right)\\
        &\le Q_{I_n}\left( \int_{\mathbb{C}^n}(\omega^{-1} \chi_k)(z-\delta w)\rho(w)dw
        \right)\\
        &\le \int_{\mathbb{C}^n}\left(Q_{I_n}(\omega^{-1}\chi_k)\right)(z-\delta w)\rho(w)dw\\
        &\le \left(\int_{B_1}\left|\left(Q_{I_n}(\omega^{-1}\chi_k)\right)(z-\delta w)-c\right|^2 dw\right)^{1/2}\|\rho\|_2 +c\\
        &\le C_\delta \|Q_\omega(\chi_k)-c\|_{L^2(\omega)} +c,
        \end{split}
    \end{equation}
    where $C_\delta$ is a constant depending on $\delta$ and $B_1$ is a unit ball with a center zero. In the second equality, we used the fact that $\omega_0$ has constant coefficients. In the first inequality, we used that $\omega_0\le \omega$. In the second inequality, we used that $Q_{I_n}$ is convex. When $k$ tends to $\infty$, the last line tends to $c$. We claim that $Q_{\omega_0}(\delb u^\delta_\infty)\le_v c$ on $U_{0,\delta}$. Indeed, note that $u^\delta_k$ converges to $u^\delta_\infty$ uniformly in $U_{0,\delta}$. Therefore, a sequence $\{ u^\delta_k\}$ and the limit $u^\delta_\infty$ satisfy the conditions in Proposition \ref{propCIL}. Choose a point $z \in U_{0,\delta}$ and an upper test $q$ of $u^\delta_\infty$ at $z$. By Proposition \ref{propCIL}, there exist points $\hat{x}_k$ and upper tests $q_k$ of $u^\delta_k$ at them such that the assertion satisfies. It implies 
    $$Q_{\omega_0}(\delb q_k)(\hat{x}_k) \rightarrow Q_{\omega_0}(\delb q)(z) \le c,$$
    which yields $Q_{\omega_0}(\delb u^\delta_\infty)\le_v c$ on $U_{0,\delta}$.
    Next, we let $\delta$ tend to zero. Since $u_\infty = f + \psi_\infty$ is psh, a sequence $\{ u^\delta_\infty \}$ is decreasing. This property implies that the sequence satisfies the conditions in Proposition \ref{propCIL}. 
    Indeed, condition \ref{itemCIL1} is obvious by taking the same point as a sequence. For condition \ref{itemCIL2}, note that for any fixed $\delta_0>0$, the function $u^{\delta_0}_\infty$ is upper semicontinuous, so for any $\delta < \delta_0$ and $z_\delta \in V \subset\subset U_0$ satisfying $z_\delta \rightarrow x \in V$,
    $$\limsup{u^\delta_\infty}(z_\delta)\le \limsup{u^{\delta_0}_\infty}(z_\delta) \le u^{\delta_0}_\infty(x).$$
    Letting $\delta_0$ tend to zero, we get condition \ref{itemCIL2}. Therefore we can apply Proposition \ref{propCIL} again to get
    $$Q_{\omega_0}(\delb u_\infty) \le_v c \ \text{on $U_0$}.$$
    Especially at point $p$, for any upper test $q$ of $u_\infty$ at $p$, we have
    $$Q_{\omega_0}(\delb q)(p) \le c.$$
    By taking a sequence of K{\"a}hler forms $\{\omega_i\}$ with constant coefficients such that $\omega_i \le \omega$ and $\omega_i(p) \rightarrow \omega(p)$ (by taking $U_0$ smaller and smaller), we get
    \[Q_\omega(\delb q)(p) \le c. \qedhere \]
\end{proof}

\begin{remark}
    The arguments in the proof clarify that the condition on the positivity of currents introduced in \cite[Definition 3.3]{GChen} implies the condition of viscosity subsolutions. Namely, if a closed positive $(1,1)$-current $\chi$ is in $\bar{\Gamma}_{\omega,c}$, then
    $P_\omega(\chi)\le_v c.$ Here, a closed positive $(1,1)$-current $\chi$ is in $\bar{\Gamma}_{\omega,c}$ if for any open subset $O$ of any coordinate neighborhood and any K{\"a}hler form $\omega_0$ on $O$ with constant coefficients such that $\omega_0 \le \chi$, we have
    $$P_{\omega_0}(\delb \varphi_\delta) \le c$$
    for any $\delta>0$,
    where $\varphi$ is a local potential of $\chi$ and $\varphi_\delta$ is its regularization via convolution.
    By standard techniques in the viscosity theory, we can prove that they are equivalent.
    \begin{prop}\label{propequivalence}
        A closed positive  $(1,1)$-current $\chi$ satisfies $P_\omega(\chi)\le_v c$ if and only if $\chi \in \bar{\Gamma}_{\omega,c}$.
    \end{prop}
    \begin{proof}
        We follow the arguments in \cite[Proposition 1.11 and Theorem 1.9]{EGZ}. Assume $P_\omega(\chi)\le_v c$.
        Take an open subset $O$ of a coordinate neighborhood and a K{\"a}hler form $\omega_0$ with constant coefficients on $O$ such that $\omega_0\le\omega$. Note that an upper test of a psh function is psh as in the proof of \cite[Proposition 1.3]{EGZ}. Thus, we have $P_{\omega_0}(\chi)\le_v c$ on $O$. Let $\varphi$ and $\psi$ be upper semicontinuous local potentials of $\chi$ and $\omega$ respectively. Define $\varphi_k:=\sup\{\varphi,a\psi-k\}$, where $k$ is a positive integer and $a$ is a real number large enough to satisfy $P_\omega(a\omega)\le c$. Since the supremum of a finite number of viscosity subsolutions is a viscosity subsolution, we have $P_{\omega_0}(\varphi_k)\le_v c$ on $O$. Denote by $\varphi_k^\varepsilon$ the sup convolution of $\varphi_k$, i.e.
        $$\varphi_k^\varepsilon(x) = \sup \left\{ \varphi_k(y)-\frac{1}{2\varepsilon^2}|x-y|^2 \, \relmiddle{|} y \in O \right\}$$
        for $x\in O_{A_k \varepsilon}$, where $O_{A_k \varepsilon}:=\{x\in O\mid \mathrm{dist}(x,\partial O)<A_k \varepsilon\}$ and $A_k$ is a real number such that $A_k^2>2\, \mathrm{osc}_O \, \varphi_k$ which is well defined since $\varphi_k$ is bounded. By this definition of $A_k$, the supremum is attained in $B(x, A_k \varepsilon)$.
        Since $\omega_0$ has constant coefficients, from the proof of \cite[Proposition 4.2]{Ishii}, we see that a function $\varphi_k^\varepsilon$ satisfies $P_{\omega_0}(\delb \varphi_k^\varepsilon)\le_v c$ on $O_{A_k\varepsilon}$.
        Note that $\varphi_k^\varepsilon$ is semiconvex since the supremum of semiconvex functions is semiconvex. By Aleksandrov's theorem, $\varphi_k^\varepsilon$ is twice differentiable almost everywhere. 
        Thus, by the convexity of $P_{\omega_0}$, for $z \in O_{A_k\varepsilon+\delta}$, we have 
        \begin{align*}
            P_{\omega_0}(\delb(\varphi_k^\varepsilon)_\delta) (z)
            &= P_{\omega_0}\left(\int_
            {B_1} (\delb\varphi_k^\varepsilon)(z-\delta w)\rho(w)dw\right) \\
            &\le \int_ {B_1} 
            P_{\omega_0}(\delb\varphi_k^\varepsilon)(z-\delta w )\rho(w)dw \\
            &\le c.
        \end{align*}
        Since a sequence $\{\varphi_k^\varepsilon\}$ decreases to $\varphi_k$, the Lebesgue theorem implies that $\varphi_k^\varepsilon$ converges to $\varphi_k$ in $L^1$. Therefore, a sequence $\{(\varphi_k^\varepsilon)_\delta\}$ decreases to $(\varphi_k)_\delta$. By the same arguments as above, Proposition \ref{propCIL} confirms $P_{\omega_0}(\delb(\varphi_k)_\delta) \le c$ in $O_\delta$. By applying the same arguments to a sequence $\{(\varphi_k)_\delta\}_k$, as $k$ tends to $\infty$, we get $P_{\omega_0}(\delb\varphi_\delta)\le c$, which means $\chi\in \bar{\Gamma}_{\omega,c}$.
    \end{proof}
\end{remark}

\begin{remark}\label{remL1}
    The arguments in the proof of Proposition \ref{propvsub} also imply that an $L^1$-limit of viscosity subsolutions is a viscosity subsolution. This property was remarked in \cite[Remark 3.4]{GChen} in terms of $\bar{\Gamma}_{\omega,c}$.
\end{remark}

\subsection{The convergence to the weak solution}\label{secpp}
We expect that the viscosity subsolution obtained in Subsection \ref{secvsub} is a viscosity solution in some sense (e.g., \cite[Definition 3.2]{EGZ3}). Also, we expect that the viscosity solution is unique and equivalent to the weak solution in terms of non-pluripolar products (e.g. \cite{EGZ, EGZ2}). However, we were unable to prove these statements at this point. 

In dimension 2, the $J$-equation is equivalent to the Monge-Amp{\`e}re equation.
We prove that the viscosity subsolution obtained in Subsection \ref{secvsub} is the solution of the Monge-Amp{\`e}re equation constructed in \cite{BEGZ}.  
To prove this, we compare viscosity subsolutions and pluripotential subsolutions (i.e. in terms of non-pluripolar products) of the Monge-Amp{\`e}re equation. 
\cite[Corollary 2.6]{EGZ} claims the following:

\begin{lem}[{\cite[Corollary 2.6]{EGZ}}]\label{lemEGZ}
    Let $X$ be a $n$-dimensional compact K{\"a}hler manifold and $v$ a volume form with nonnegative continuous density. Let $\theta$ be a smooth closed real $(1,1)$-form whose cohomology class is big. Then a $\theta$-psh function $\varphi$ satisfies $(\theta+\delb\varphi)^n \ge v$ in the viscosity sense if and only if $\langle(\theta+\delb\varphi)^n\rangle\ge v$, where $\langle\cdot\rangle$ denotes the non-pluripolar product.
\end{lem}
\cite{EGZ} states that the lemma is demonstrated using a method similar to that of \cite[Theorem 1.9]{EGZ}, which asserts locally. However, here we provide a detailed proof of the ``only if'' part under the following assumption.
\begin{asmp}\label{asmp}
     There exists a $\theta$-psh function $\psi$ such that $(\theta+\delb \psi)^n\ge v$ in the viscosity sense on $\mathrm{Amp}(\theta)$ and $\psi$ has minimal singularities.
\end{asmp}

\begin{proof}[Proof of the ``only if'' part  of Lemma \ref{lemEGZ} under Assumption \ref{asmp}]
    Define $\varphi^{(k)}=\sup \{ \varphi, \psi-k \}$, where $\varphi$ and $\psi$ are the functions of the assumptions.  As the condition of viscosity subsolutions holds after taking the supremum of a finite number of them, we see that $(\theta+\delb{\varphi^{(k)}})^n \ge v$ in the viscosity sense on $\mathrm{Amp}(\theta)$. By assumption, the function $\varphi^{(k)}$ is locally bounded on $\mathrm{Amp}(\theta)$. Since $\mathrm{Amp}(\theta)$ is open, by the local method \cite[Theorem 1.9]{EGZ}, we get
    $$\langle(\theta+\delb\varphi^{(k)})^n\rangle\ge v \quad \text{on $\mathrm{Amp}(\theta)$}.$$
    Since $v$ puts no mass on $\mathrm{Amp}^c(\theta)$, we get
    $$\langle(\theta+\delb\varphi^{(k)})^n\rangle\ge v \quad \text{on $X$}.$$
    Recall that the non-pluripolar product can be characterized by
    $$\langle(\theta+\delb\varphi)^n\rangle = \lim_{k\rightarrow\infty}\mathbbm{1}_{\{\varphi >\varphi_{min}-k\}}\langle(\theta+\delb\sup \{ \varphi, \varphi_{min}-k \})^n\rangle,$$
    where $\varphi_{min}$ is a $\theta$-psh function with minimal singularities.
    If we put $\varphi_{min}=\psi$, the statement follows.
\end{proof}

We use the solution constructed in \cite{BEGZ}, which is known to satisfy Assumption \ref{asmp} by \cite[Theorems B and C]{BEGZ}.

\begin{prop}
    Let $X$ be a compact K{\"a}hler surface and a pair of K{\"a}hler classes $([\chi],[\omega])$ on $X$ be $J$-nef. If a positive current $T \in [\chi]$ satisfies $Q_\omega(T)\le_v c$, then $c\, T-\omega$ is a positive current and
    $$\langle (c\, T-\omega)^2\rangle=\omega^2.$$
\end{prop}

\begin{proof}
    Since $P_\omega \le Q_\omega$ and an upper test of a psh function is psh as in the proof of \cite[Proposition 1.3]{EGZ}, we have $P_\omega(T)\le_v c$. By Proposition \ref{propequivalence}, we have $T \in \bar{\Gamma}_{\omega,c}$. Equivalently, the current $c\, T-\omega$ is a positive current since $\delb\varphi_\delta \rightarrow \delb\varphi=T$ as $\delta \rightarrow 0$. 
    We can also see that $Q_\omega(T) \le_v c$ is equivalent to $(c\, T-\omega)^2 \ge \omega^2$ in the viscosity sense. 
    By using Lemma \ref{lemEGZ}, we see that $\langle(c \, T-\omega)^2 \rangle \ge \omega^2$. 
    On the other hand, by \cite[Proposition 1.20]{BEGZ}, we have
    $$\int_X \langle(c\, T-\omega)^2\rangle \le \langle[c\chi-\omega]^2\rangle = \int_X\omega^2,$$ 
    where for the equality, we use the assumption that $[c\chi-\omega]$ is nef. Hence, we obtain $\langle(c \, T-\omega)^2\rangle =\omega^2$.
\end{proof}

By combining the propositions we have obtained so far and the uniqueness of the solution of the Monge-Amp{\`e}re equation \cite[Theorem A]{BEGZ}, we get the main theorem.

\section{The dHYM flow}\label{secdHYM}
We prove Theorem \ref{thmdHYMmain} using the same strategy as in the previous section. In the whole section, we assume $0<\theta_0<\pi$. We start with a variant of Lemma \ref{lemconvexity}. Recall that the $\mathcal{J}$-functional in the context of the dHYM equation is defined by
\begin{equation*}
    \mathcal{J}(0)=0, \quad d\mathcal{J}(\varphi)(\psi)=\int_X\psi \, \mathrm{Im}\left(e^{-\sqrt{-1}\theta_0}(\alpha_\varphi+\sqrt{-1}\omega)^n\right).
\end{equation*}

\begin{lem}\label{lemdhymconvexity}
    The $\mathcal{J}$-functional is convex along the dHYM flow with an initial point $\varphi_0$ satisfying $0<\theta_\omega(\alpha_0)<\pi$.
\end{lem}

\begin{remark}
    An initial point as in the statement exists since the same proof as \cite[Lemma 5.2]{FYZ} works for our case. 
\end{remark}

\begin{proof}[Proof of Lemma \ref{lemdhymconvexity}]
    We can compute as
    \begin{align*}
        \frac{d^2}{dt^2}\mathcal{J}(\varphi)
        &= \frac{d}{dt}\int_X \dot{\varphi} \, \mathrm{Im}\left(e^{-\sqrt{-1}\theta_0}(\alpha_\varphi+\sqrt{-1}\omega)^n\right) \\
        &= \sin{\theta_0} \frac{d}{dt}\int_X \dot{\varphi} \left(\cot{\theta_0} \mathrm{Im}(\alpha_\varphi+\sqrt{-1}\omega)^n - \mathrm{Re}(\alpha_\varphi+\sqrt{-1}\omega)^n\right) \\
        &=-\sin{\theta_0}\frac{d}{dt}\int_X \dot{\varphi}^2 \, \mathrm{Im}(\alpha_\varphi+\sqrt{-1}\omega)^n \\
        &= -\sin{\theta_0} \left\{ \int_X 2\, \dot{\varphi}\, \ddot{\varphi} \, \mathrm{Im}(\alpha_\varphi+\sqrt{-1}\omega)^n \right. \\
        &\hspace{2.5cm} \left. +\int_X \dot{\varphi}^2 n \, \delb\dot{\varphi} \wedge \mathrm{Im}(\alpha_\varphi+\sqrt{-1}\omega)^{n-1} \right\}\\
        &= -2\sin{\theta_0} \left\{ \int_X  \dot{\varphi}\, \ddot{\varphi} \, \mathrm{Im}(\alpha_\varphi+\sqrt{-1}\omega)^n \right. \\
        &\hspace{2.5cm} \left. -\int_X \dot{\varphi} \, n\sqrt{-1}\partial\cot{\theta_\omega(\alpha_\varphi)} \wedge \bar{\partial}\dot{\varphi}\wedge \mathrm{Im}(\alpha_\varphi+\sqrt{-1}\omega)^{n-1} \right\}.
    \end{align*}
    By the definition of the dHYM flow \eqref{eqdhymflow}, we also have
    \begin{align*}
        \ddot{\varphi} 
        &= \frac{\frac{d}{dt}\mathrm{Re}(\alpha_\varphi+\sqrt{-1}\omega)^n}{\mathrm{Im}(\alpha_\varphi+\sqrt{-1}\omega)^n}-\cot{\theta_\omega(\alpha_\varphi)}\frac{\frac{d}{dt}\mathrm{Im}(\alpha_\varphi+\sqrt{-1}\omega)^n}{\mathrm{Im}(\alpha_\varphi+\sqrt{-1}\omega)^n}\\
        &=\frac{n\delb\dot{\varphi}\wedge\mathrm{Re}(\alpha_\varphi+\sqrt{-1}\omega)^{n-1}}{\mathrm{Im}(\alpha_\varphi+\sqrt{-1}\omega)^n} \\
        &\quad -\cot{\theta_\omega(\alpha_\varphi)}\frac{n\delb\dot{\varphi}\wedge\mathrm{Im}(\alpha_\varphi+\sqrt{-1}\omega)^{n-1}}{\mathrm{Im}(\alpha_\varphi+\sqrt{-1}\omega)^n}.
    \end{align*}
    By inserting this to the above, we get
    \begin{align*}
        \frac{d^2}{dt^2}\mathcal{J}(\varphi)
        &= -2\sin{\theta_0} \left\{ \int_X  \dot{\varphi}\, \left( n\delb\dot{\varphi}\wedge\mathrm{Re}(\alpha_\varphi+\sqrt{-1}\omega)^{n-1}  \right. \right. \\
        &\hspace{3cm} \left. \left. - \cot{\theta_\omega(\alpha_\varphi)} n\delb\dot{\varphi}\wedge\mathrm{Im}(\alpha_\varphi+\sqrt{-1}\omega)^{n-1}
        \right) \right. \\
        &\hspace{2.5cm} \left. -\int_X \dot{\varphi} \, n\sqrt{-1}\partial\cot{\theta_\omega(\alpha_\varphi)} \wedge \bar{\partial}\dot{\varphi}\wedge \mathrm{Im}(\alpha_\varphi+\sqrt{-1}\omega)^{n-1} \right\}.
        \end{align*}
    Applying integration by parts to the first and second terms, we get
    \begin{align*}
        \frac{d^2}{dt^2}\mathcal{J}(\varphi)
        &= 2\sin{\theta_0} \left( \int_X n\, \sqrt{-1}\partial\dot{\varphi}\wedge\bar{\partial}\dot{\varphi}\right.\\
        &\hspace{2cm}\biggl. \wedge \left( \mathrm{Re}(\alpha_\varphi+\sqrt{-1}\omega)^{n-1}- \cot{\theta_\omega(\alpha_\varphi)}\mathrm{Im}(\alpha_\varphi+\sqrt{-1}\omega)^{n-1}\right)\biggr).
    \end{align*}
    We claim that 
    $$ n\sqrt{-1}\partial\dot{\varphi}\wedge \bar{\partial}\dot{\varphi}\wedge \left( \mathrm{Re}(\alpha_\varphi+\sqrt{-1}\omega)^{n-1}- \cot{\theta_\omega(\alpha_\varphi)}\mathrm{Im}(\alpha_\varphi+\sqrt{-1}\omega)^{n-1}\right) \ge0.$$
    Indeed, take a point $p \in X$ and a coordinate which is normal for $\omega$ at $p$ and on which $\alpha_\varphi(p)=\sum_i \lambda_i \sqrt{-1}dz^i\wedge d\bar{z}^i$. Then, at $p$, we have
    \begin{align*}
        &n\sqrt{-1}\partial\dot{\varphi}\wedge \bar{\partial}\dot{\varphi}\wedge \left( \mathrm{Re}(\alpha_\varphi+\sqrt{-1}\omega)^{n-1}- \cot{\theta_\omega(\alpha_\varphi)}\mathrm{Im}(\alpha_\varphi+\sqrt{-1}\omega)^{n-1}\right)\\
        =\, &n!\left\{ \mathrm{Re}\Biggl( \left(\sum_i \frac{\partial_i\dot{\varphi} \, \partial_{\bar{i}}\dot{\varphi}}{\lambda_i+\sqrt{-1}}\right)\prod_i(\lambda_i+\sqrt{-1})\sqrt{-1}dz^i\wedge d\bar{z}^i\Biggr)\right.\\
        &\left. \quad -\frac{\mathrm{Re}\prod_i\lambda_i+\sqrt{-1}}{\mathrm{Im}\prod_i\lambda_i+\sqrt{-1}}\mathrm{Im}\left(\left(\sum_i \frac{\partial_i\dot{\varphi} \, \partial_{\bar{i}}\dot{\varphi}}{\lambda_i+\sqrt{-1}}\right)\prod_i(\lambda_i+\sqrt{-1})\sqrt{-1}dz^i\wedge d\bar{z}^i\right)\right\} \\
        =\, &n!\left\{\left(\sum_i\frac{\lambda_i\, \partial_i\dot{\varphi}\, \partial_{\bar{i}}\dot{\varphi}}{\lambda_i^2+1}\right) \mathrm{Re}\prod_i(\lambda_i+\sqrt{-1}) +\left(\sum_i\frac{\partial_i\dot{\varphi}\, \partial_{\bar{i}}\dot{\varphi}}{\lambda_i^2+1}\right)\mathrm{Im}\prod_i(\lambda_i+\sqrt{-1})\right.\\
        &\left. \quad -\frac{\mathrm{Re}\prod_i\lambda_i+\sqrt{-1}}{\mathrm{Im}\prod_i\lambda_i+\sqrt{-1}}\left(\sum_i\frac{\lambda_i\, \partial_i\dot{\varphi}\, \partial_{\bar{i}}\dot{\varphi}}{\lambda_i^2+1}\right)\mathrm{Im}\prod_i(\lambda_i+\sqrt{-1})\right.\\
        &\left. \quad +\frac{\mathrm{Re}\prod_i\lambda_i+\sqrt{-1}}{\mathrm{Im}\prod_i\lambda_i+\sqrt{-1}} \left(\sum_i\frac{\partial_i\dot{\varphi}\, \partial_{\bar{i}}\dot{\varphi}}{\lambda_i^2+1}\right)\mathrm{Re}\prod_i(\lambda_i+\sqrt{-1})
        \right\}\prod_i \sqrt{-1}dz^i\wedge d\bar{z}^i \\
        = &n!\left( \sum_i \frac{\partial_i\dot{\varphi}\, \partial_{\bar{i}}\dot{\varphi}}{\lambda_i^2+1}\right) \frac{\prod_i(\lambda_i^2+1)}{\left(\mathrm{Im}\prod_i\lambda_i+\sqrt{-1}\right)}\prod_i \sqrt{-1}dz^i\wedge d\bar{z}^i.
    \end{align*}
    Note that if $ 0 < \theta_\omega(\alpha_{\varphi_0})<\pi$, then $0 < \theta_\omega(\alpha_\varphi)< \pi$ for all $t \in [0,\infty)$ by \cite[Lemma 3.2]{FYZ}. In particular, we have
    \begin{equation*}
        \mathrm{Im}\prod_i\lambda_i+\sqrt{-1}=\left(\prod_i\lambda_i^2+1\right)\sin\theta_\omega(\alpha_\varphi)>0.
    \end{equation*}
    Therefore, we conclude that the $\mathcal{J}$-functional is convex along the dHYM flow if the initial point satisfies $0 <\theta_\omega(\alpha_{\varphi_0}) < \pi$.
\end{proof}

\begin{prop}
    Let $X$ be a compact K{\"a}hler surface and $\omega$ a K{\"a}hler form. Let $\alpha$ be a smooth closed real $(1,1)$-form. Assume that $0 <\theta_0 < \pi$ and the class $[\alpha-\cot{\theta_0}\omega]$ is nef. If an initial point of the dHYM flow satisfies $0 < \theta_\omega(\alpha_{\varphi_0}) < \pi$, then the time derivative of the dHYM flow converges to zero in $L^2$.
\end{prop}

\begin{proof}
    Note that
    \begin{align}\label{eqdhymderiv}
        \frac{d}{dt}\mathcal{J}(\varphi_t)
        &=\int_X \dot{\varphi_t} \, \mathrm{Im}\left(e^{-\sqrt{-1}\theta_0}(\alpha_\varphi+\sqrt{-1}\omega)^n\right) \notag\\
        &=\sin{\theta_0}\int_X \dot{\varphi}_t \left(\cot{\theta_0}\, \mathrm{Im}(\alpha_\varphi+\sqrt{-1}\omega)^n-\mathrm{Re}(\alpha_\varphi+\sqrt{-1}\omega)^n
        \right) \notag\\
        &=-\sin{\theta_0}\int_X \dot{\varphi}_t^2 \, \mathrm{Im}(\alpha_\varphi+\sqrt{-1}\omega)^n \le0.
    \end{align}
    As in \cite{CL}, let $F_{\omega,-\varepsilon}$ be the operator defined by
    $$F_{\omega,-\varepsilon}(\alpha_\varphi)= \frac{\mathrm{Re}(\alpha_\varphi+\sqrt{-1}\omega)^n-\varepsilon\omega^n}{\mathrm{Im}(\alpha_\varphi+\sqrt{-1}\omega)^n}.$$
    Also, let $\mathcal{J}_{-\varepsilon}$ be the functional defined by
    $$\mathcal{J}_{-\varepsilon}(0)=0, \quad d\mathcal{J}_{-\varepsilon}(\varphi)(\psi)=\int_X \psi \, \left(\cot{\theta_0}-a_0\varepsilon-F_{\omega,-\varepsilon}(\alpha_\varphi)\right)\mathrm{Im}(\alpha_\varphi+\sqrt{-1}\omega)^n,$$
    where $a_0$ is the constant determined by
    $$a_0=\frac{\int_X\omega^n}{\int_X \mathrm{Im}(\alpha+\sqrt{-1}\omega)^n}.$$
    We claim that if the class $[\alpha-\cot{\theta_0}\omega]$ is nef and big, then for any $\varepsilon>0$, the functional $\mathcal{J}_{-\varepsilon}$ is bounded from below on 
    $$\mathcal{H}_\pi=\{ \varphi \in C^\infty(X) \mid  0 < \theta_\omega(\alpha_\varphi) < \pi\}.$$
    Note that the critical point $\varphi$ of $\mathcal{J}_{-\varepsilon}$ satisfies
    $$F_{\omega,-\varepsilon}(\alpha_\varphi)=\cot{\theta_0}-a_0\varepsilon,$$
    which is equivalent to
    \begin{equation}\label{eqMA}
        \left(\alpha_\varphi-(\cot{\theta_0}-a_0\varepsilon)\omega\right)^2=(\left(\cot{\theta_0}-a_0\varepsilon\right)^2+1+\varepsilon)\omega^2.
    \end{equation}
    Since $[\alpha-\cot{\theta_0}\omega]$ is nef and $\omega$ is a K{\"a}hler form, the class $[\alpha-(\cot{\theta_0}-a_0\varepsilon)\omega]$ is K{\"a}hler. By Yau's theorem \cite{Yau}, there exists a unique solution $u_{MA}$ of \eqref{eqMA} such that $\mathrm{Im}\mathrm{CY}_\mathbb{C}(u_{MA})=0$. Here, recall that the functional $\mathrm{CY}_\mathbb{C}$ is given by
    $$\mathrm{CY}_\mathbb{C}(0)=0,\quad d\mathrm{CY}_\mathbb{C}(\varphi)(\psi)=\int_X \psi \, (\alpha_\varphi+\sqrt{-1}\omega)^n.$$
    With the properties of $F_{\omega,-\varepsilon}$ \cite[Lemma 5.6]{GChen}, the same arguments as \cite[Section 4.2]{CL} imply that the twisted dHYM flow
    \begin{equation}\label{eqtdHYMflow}
        \begin{cases}
        \dot{u}_t=F_{\omega,-\varepsilon}(\alpha_u)-\cot{\theta_0}+a_0\varepsilon \\ u(x,0)=u_0(x)\in \mathcal{H}_\pi
        \end{cases}
    \end{equation}
    converges smoothly to $u_{MA}+\mathrm{Im}\mathrm{CY}_\mathbb{C}(u_0)$. Since the flow is a gradient flow of $\mathcal{J}_{-\varepsilon}$, we have
    $$\mathcal{J}_{-\varepsilon}(u_0)\ge\mathcal{J}_{-\varepsilon}(u_{MA})=-C_\varepsilon,$$
    namely, the claim follows.
    Then, along the dHYM flow $\varphi_t$ with the initial point $\varphi_0$ satisfying $0<\theta_\omega(\alpha_{\varphi_0}) < \pi$, we have
    \begin{equation}\label{ineqbdd}
        \mathcal{J}_{-\varepsilon}(\varphi_t)\ge - C_\varepsilon
    \end{equation}
    for some constant $C_\varepsilon$, since 
    $\inf\theta_\omega(\alpha_{\varphi_0})<\theta_\omega(\alpha_{\phi_t})<\sup\theta_\omega(\alpha_{\varphi_0})$
    by \cite[Lemma 3.2]{FYZ}. Recall that the imaginary part of the Calabi-Yau functional $\mathrm{CY}_\mathbb{C}$ 
    is constant along the dHYM flow by \cite[Lemma 2.7]{FYZ}. Thus,
    \begin{align*}
        \mathcal{J}_{-\varepsilon}(\varphi_t)-\mathcal{J}_{-\varepsilon}(\varphi_0)
        &=\int^t_0\int_X \dot{\varphi}_t \, \biggl( (\cot{\theta_0}-\cot{\theta_\omega(\alpha_\varphi)})\mathrm{Im}(\alpha_\varphi+\sqrt{-1}\omega)^n+\varepsilon\, \omega^n\biggr)\\
        &=\frac{1}{\sin{\theta_0}}\mathcal{J}(\varphi_t)+\varepsilon\int_X\varphi_t\, \omega^n\\
        &\le\frac{1}{\sin{\theta_0}}\mathcal{J}(\varphi_t)+\varepsilon C\sup\varphi_t\\
        &\le\frac{1}{\sin{\theta_0}}\mathcal{J}(\varphi_t)+\varepsilon C t,
    \end{align*}
    where we used \cite[Lemma 3.2]{FYZ} again to estimate $\sup\varphi_t$.
    By combining this with \eqref{ineqbdd}, we see that for any $\varepsilon>0$ there exist constants $C$ and $C_\varepsilon$ such that
    \begin{equation}\label{ineqdhymslope}
        \mathcal{J}(\varphi_t)\ge -\varepsilon C t-C_\varepsilon.
    \end{equation}
    By Lemma \ref{lemdhymconvexity}, there exists a constant $\mu\le0$ such that
    $$\lim_{t\rightarrow\infty}\frac{d}{dt}\mathcal{J}(\varphi_t)=\mu.$$
    By dividing both hands of \eqref{ineqdhymslope} by $t$ and letting $t$ tend to $\infty$, we also see that
    $$\lim_{t\rightarrow\infty}\frac{d}{dt}\mathcal{J}(\varphi_t)\ge - \varepsilon C.$$
    Since $\varepsilon>0$ is arbitrary, we get $\mu\ge0$, which concludes $\mu=0$. From \eqref{eqdhymderiv}, we obtain
    $$\int_X \dot{\varphi}_t^2 \, \mathrm{Im}(\alpha_\varphi+\sqrt{-1}\omega)^n \rightarrow0.$$
    By \cite[Lemma 3.2]{FYZ}, we have
    \begin{align*}
        \mathrm{Im}(\alpha_\varphi+\sqrt{-1}\omega)^n
        &=\left(\prod_i\sqrt{\lambda_i^2+1}\right)\left(\sin{\theta_\omega(\alpha_\varphi)}\right)\omega^n\\
        &\ge \inf\left(\sin{\theta_\omega(\alpha_{\varphi_0}})\right) \omega^n.
    \end{align*}    
    Therefore, we conclude the statement.
\end{proof}

The proposition settles the dHYM variant of Proposition \ref{propslope} in dimension 2. We see below that the arguments in Subsection \ref{secvsub} can apply to the dHYM flow as well. 
For a K{\"a}hler form $\omega$ and a smooth closed real $(1,1)$-form $\alpha$, we denote by $P$ and $Q$ the operators defined by
$$P_\omega(\alpha)=-\min_{k=1,\dots,n}\cot{\sum_{j\neq k}\mathrm{arccot}\, {\lambda_j}},\quad Q_\omega(\alpha)=-\cot{\sum_j \mathrm{arccot}\, {\lambda_j}},$$
where $\lambda_j$'s are eigenvalues of $\omega^{-1}\alpha$. The operators are convex by \cite[Lemma 5.6]{GChen}.
We say that a current $\alpha$ is a viscosity subsolution of
$$Q_\omega(\alpha)=-\cot{\theta_0},$$
denoted by
$$Q_\omega(\alpha)\le_v -\cot{\theta_0},$$
if for any point $x\in X$ and any upper test $q$ of $\varphi$ at $x$, where $\varphi$ is a upper semicontinuous local potential of $\alpha$, we have
\begin{equation*}
    Q_\omega(\delb q)(x)\le-\cot{\theta_0}.
\end{equation*}
By \cite[Lemma 3.2]{FYZ}, along the dHYM flow $\varphi_t$, we have
\begin{equation}\label{eqpsh}
    \alpha_{\varphi_t}-\left(\cot{\max \theta_\omega(\alpha_{\varphi_0})}\right)\omega>0.
\end{equation}
Therefore, the sequence of functions $\psi_t=\varphi_t-\sup\varphi_t$ has an $L^1$-limit, denoted by $\psi_\infty$, with $\sup \psi_\infty=0$ if we take a subsequence. 

\begin{prop}\label{propdhymvsub}
    Suppose that a sequence of closed real $(1,1)$-forms $\alpha_k=\alpha+\delb \psi_k$ satisfies
    \begin{align*}
        \psi_k \rightarrow \psi_\infty \, \text{in $L^1(\omega)$ and} \ 
        \|Q_\omega(\alpha_k)+\cot{\theta_0}\|_{L^2(\omega)} \rightarrow 0.    
    \end{align*}    
    Then, we have
    $$Q_\omega(\alpha_\infty)\le_v -\cot{\theta_0},$$
    where $\alpha_\infty=\alpha+\delb \psi_\infty.$
\end{prop}

\begin{proof}
    We only point out the differences from Proposition \ref{propvsub}. To do the calculation similar to \eqref{eqregularization}, recall that, as in \cite[Lemma 5.1]{CL}, for $0<\theta<\pi$, there exists $\sigma_0(n,\theta)>0$ such that the following holds: let $\alpha$ and $\omega$ are smooth closed real $(1,1)$-forms, where $\omega$ is K{\"a}hler. Suppose $\theta_\omega(\alpha)<\theta$. For any $0<\varepsilon<\sigma_0$, there exists $\delta_\varepsilon>0$ such that if $\omega_0$ is a K{\"a}hler form with constant coefficients on a coordinate neighborhood satisfying $ |\omega-\omega_0|<\delta_\varepsilon$, then
    $$Q_{\omega_0}(\alpha+\varepsilon\omega)\le Q_\omega(\alpha),\quad P_{\omega_0}(\alpha+\varepsilon\omega)\le P_\omega(\alpha).$$
    By this fact, we have
    $$Q_{\omega_0}(\delb u_{k,\varepsilon}^\delta)(z)\le C_\delta\|Q_\omega(\alpha_k)+\cot{\theta_0}\|_{L^2(\omega)}-\cot{\theta_0},$$
    where $u_{k,\varepsilon}$ is an upper semicontinuous local potential of $\alpha_{\varphi_k}+\varepsilon\omega$ and other notations are analogous to \eqref{eqregularization}. By \eqref{eqpsh}, the same observation as in the proof of Proposition \ref{propvsub} implies that the sequences $\{u^\delta_{k,\varepsilon}\}_k$ and $\{u^\delta_{\infty,\varepsilon}\}_\delta$ satisfy the conditions in Proposition \ref{propCIL}. Therefore, we get
    $$Q_\omega(\alpha_\infty+\varepsilon\omega)\le_v -\cot{\theta_0}.$$
    Since $\omega$ is a smooth form, we conclude the statement by letting $\varepsilon$ tend to zero and using Proposition \ref{propCIL} again.
\end{proof}

The dHYM version of Proposition \ref{propequivalence} is the following:
\begin{prop}\label{propdhymequivalence}
    Let $\omega$ be a K{\"a}hler form. Then a current $\alpha$ satisfies $P_\omega(\alpha)\le_v -\cot{\theta_0}$ if and only if for any $\varepsilon>0$, any open subset of a coordinate neighborhood $O$, and any K{\"a}hler form $\omega_0$ with constant coefficients on $O$ which is sufficiently close to $\omega$, we have
    $$P_{\omega_0}(\delb \varphi^\delta_\varepsilon)\le_v -\cot{\theta_0},$$
    where $\varphi_\varepsilon$ is a local potential of $\alpha+\varepsilon\omega$ and $\varphi^\delta_\varepsilon$ is its regularization.
\end{prop}

The proof is identical to that of Proposition \ref{propequivalence}, and is therefore omitted.

The arguments in Subsection \ref{secpp} follow verbatim for the dHYM flow since we have $P_\omega\le Q_\omega$, the equivalence between the viscosity subsolutions of the dHYM equation and the Monge-Amp{\`e}re equation, and Proposition \ref{propdhymequivalence}. Therefore, we complete the proof of the main theorem.


\begin{thebibliography}{99}

 \bibitem[BEGZ10]{BEGZ}
  S. Boucksom, P. Eyssidieux, V. Guedj, and A. Zeriahi, 
  \textit{Monge-Amp{\`e}re equations in big cohomology classes,}
  Acta Math. \textbf{205} (2010), no. 2, 199--262.
 \bibitem[Chen21]{GChen}
  G. Chen, 
  \textit{The J-equation and the supercritical deformed Hermitian-Yang-Mills equation,} 
  Invent. Math. \textbf{225} (2021), no. 2, 529--602. 
 \bibitem[Chen00]{XChen}
  X. X. Chen, 
  \textit{On the lower bound of the Mabuchi energy and its application,}
  Internat. Math. Res. Notices 2000, no. 12, 607--623. 
 \bibitem[Chen04]{XChen2}
  X. X. Chen, 
  \textit{A new parabolic flow in K{\"a}hler manifolds,} 
  Comm. Anal. Geom. \textbf{12} (2004), no. 4, 837--852.
 \bibitem[CL21]{CL}
  J. Chu and M.-C. Lee,
  \textit{Hypercritical deformed Hermitian-Yang-Mills equation,}
  preprint, available at arXiv:2107.13192v1.
 \bibitem[CLT24]{CLT}
  J. Chu, M.-C. Lee, and R. Takahashi, 
  \textit{A Nakai-Moishezon type criterion for supercritical deformed Hermitian-Yang-Mills equation,} 
  J. Differential Geom. \textbf{126} (2024), no. 2, 583--632.
 \bibitem[CS17]{CS17}
  T. C. Collins and G. Sz{\'e}kelyhidi, 
  \textit{Convergence of the $J$-flow on toric manifolds,} 
  J. Differential Geom. \textbf{107} (2017), no. 1, 47--81.
 \bibitem[CIL92]{CIL}
  M. G. Crandall, H. Ishii, and P.-L. Lions, 
  \textit{User's guide to viscosity solutions of second order partial differential equations}, 
  Bull. Amer. Math. Soc. (N.S.) \textbf{27} (1992), no. 1, 1--67.
 \bibitem[DMS23]{DMS}
  V. V. Datar, R. Mete, and J, Song,
  \textit{Minimal slopes and bubbling for complex Hessian equations,}
  preprint, available at arXiv:2312.03370v1.
 \bibitem[DP21]{Datar-Pingali}
  V. V. Datar and V. P. Pingali, 
  \textit{A numerical criterion for generalised Monge-Ampère equations on projective manifolds,}
  Geom. Funct. Anal. \textbf{31} (2021), no. 4, 767--814.
  \bibitem[DP04]{DP}
  J.-P. Demailly and M. Paun, 
  \textit{Numerical characterization of the K{\"a}hler cone of a compact K{\"a}hler manifold,} 
  Ann. of Math. (2) \textbf{159} (2004), no. 3, 1247--1274. 
 \bibitem[DDT19]{DDT}
  S. Dinew, H.-S. Do, and T. D. T{\^o}, 
  \textit{A viscosity approach to the Dirichlet problem for degenerate complex Hessian-type equations,}
  Anal. PDE \textbf{12} (2019), no. 2, 505--535.
 \bibitem[Don99]{Donaldson}
  S. K. Donaldson, 
  \textit{Moment maps and diffeomorphisms. Sir Michael Atiyah: a great mathematician of the twentieth century,} 
  Asian J. Math. \textbf{3} (1999), no. 1, 1--15.  
 \bibitem[EGZ11]{EGZ}
  P. Eyssidieux, V. Guedj, and A. Zeriahi, 
  \textit{Viscosity solutions to degenerate complex Monge-Amp{\`e}re equations,} 
  Comm. Pure Appl. Math. \textbf{64} (2011), no. 8, 1059--1094. 
 \bibitem[EGZ15]{EGZ3}
  P. Eyssidieux, V. Guedj, and A. Zeriahi 
  \textit{Continuous approximation of quasiplurisubharmonic functions,}
  Analysis, complex geometry, and mathematical physics: in honor of Duong H. Phong, 67--78, Contemp. Math., \textbf{644}, Amer. Math. Soc., Providence, RI, 2015.
 \bibitem[EGZ17]{EGZ2}
  P. Eyssidieux, V. Guedj, and A. Zeriahi, 
  \textit{Corrigendum: Viscosity solutions to complex Monge-Ampère equations [MR2839271],} 
  Comm. Pure Appl. Math. \textbf{70} (2017), no. 5, 815--821.
 \bibitem[FLSW14]{FLSW}
  H. Fang, M. Lai, J. Song, and B. Weinkove,
  \textit{The $J$-flow on Kähler surfaces: a boundary case,}
  Anal. PDE \textbf{7} (2014), no. 1, 215--226.
 \bibitem[FYZ24]{FYZ}
  J. Fu, S.-T. Yau, and D. Zhang, 
  \textit{A new flow solving the LYZ equation in Kähler geometry,} 
  J. Differential Geom. \textbf{128} (2024), no. 1, 153--192.
  \bibitem[Has19]{Hashimoto}
  Y. Hashimoto,
  \textit{Existence of twisted constant scalar curvature K{\"a}hler metrics with a large twist,}
  Math. Z. \textbf{292} (2019), no. 3-4, 791--803.
 \bibitem[Ish89]{Ishii}
  H. Ishii, 
  \textit{On uniqueness and existence of viscosity solutions of fully nonlinear second-order elliptic PDEs,} 
  Comm. Pure Appl. Math. \textbf{42} (1989), no. 1, 15--45.
 \bibitem[JY17]{JY}
  A. Jacob and S.-T. Yau, 
  \textit{A special Lagrangian type equation for holomorphic line bundles,} Math. Ann. \textbf{369} (2017), no. 1-2, 869--898.
 \bibitem[LYZ01]{LYZ}
  N. C. Leung, S.-T. Yau, and E. Zaslow, 
  \textit{From special Lagrangian to Hermitian-Yang-Mills via Fourier-Mukai transform,} 
  Winter School on Mirror Symmetry, Vector Bundles and Lagrangian Submanifolds (Cambridge, MA, 1999), 209--225, AMS/IP Stud. Adv. Math., \textbf{23}, Amer. Math. Soc., Providence, RI, 2001. 
 \bibitem[SD20]{SD}
  Z. Sj{\"o}str{\"o}m Dyrefelt,
  \textit{Optimal lower bounds for Donaldson's J-functional,}
  Adv. Math. \textbf{374} (2020), 107271, 37 pp.  
 \bibitem[Song20]{Song}
  J. Song,
  \textit{Nakai-Moishezon criterions for complex Hessian equations,}
  preprint, available at arXiv:2012.07956v1.
 \bibitem[SW08]{SW}
  J. Song and B. Weinkove, 
  \textit{On the convergence and singularities of the $J$-flow with applications to the Mabuchi energy,} 
  Comm. Pure Appl. Math. \textbf{61} (2008), no. 2, 210--229. 
 \bibitem[Yau78]{Yau}
  S.-T. Yau, 
  \textit{On the Ricci curvature of a compact Kähler manifold and the complex Monge-Ampère equation. I.} 
  Comm. Pure Appl. Math. \textbf{31} (1978), no. 3, 339--411.
\end{thebibliography}
\end{document}